\documentclass{amsart}[12pt]
 \usepackage{amsmath ,  amsthm ,  amscd ,  amsfonts,}
 \usepackage{graphicx,float}

 \usepackage{amssymb}

 \usepackage{pb-diagram}
\usepackage{lamsarrow}
\usepackage{pb-lams}

 \newcommand{\PP}{{\mathbb{P}}}

 \newcommand{\gl}{\lambda}

 \DeclareMathOperator{\dom}{dom}

 \DeclareMathOperator{\Suc}{Suc}
 \DeclareMathOperator{\Lev}{Lev}

 \def\k{\kappa}

 \def\l{\lambda}

 \def\a{\alpha}

 \newtheorem{theorem}{Theorem}[section]
 \newtheorem{lemma}[theorem]{Lemma}
 \newtheorem{proposition}[theorem]{Proposition}
 
 \newtheorem{definition}[theorem]{Definition}

 \newtheorem{remark}[theorem]{Remark}
 
 \newtheorem{main theorem}[theorem]{Main Theorem}
 
 \newtheorem{question}[theorem]{Question}
 \numberwithin{equation}{section}

 \usepackage{pb-diagram}

 \def\l{\lambda}

 \def\rmark{\mbox{$\rm\bf\rule{0.06em}{1.45ex}\kern-0.05em R$}}
 \def\pmark{\mbox{$\rm\bf\rule{0.06em}{1.45ex}\kern-0.05em P$}}
 \def\nmark{\mbox{$\rm\bf\rule{0.06em}{1.45ex}\kern-0.05em N$}}
 \def\vdash{\mbox{$\rm\| \kern-0.13em  - $}}

 \def\l{\lambda}

 \def\rmark{\mbox{$\rm\bf\rule{0.06em}{1.45ex}\kern-0.05em R$}}
 \def\pmark{\mbox{$\rm\bf\rule{0.06em}{1.45ex}\kern-0.05em P$}}
 \def\nmark{\mbox{$\rm\bf\rule{0.06em}{1.45ex}\kern-0.05em N$}}
 \def\vdash{\mbox{$\rm\| \kern-0.13em  - $}}

\begin{document}

 \title[The definable tree property for successors of cardinals]{The definable tree property \\for successors of cardinals}

 \author[A .  S .  Daghighi, M. Pourmahdian]{Ali Sadegh Daghighi$^{1}$, Massoud Pourmahdian$^{2~\dagger}$}

 \thanks{$^\dagger$The authors would like to thank Mohammad Golshani for helpful discussions during writing this paper and generous sharing of his ideas regarding the proof of the theorem \ref{definable tree property at a singular}.\\
$^1$Department of Mathematics and Computer Science, Amirkabir University of Technology, Hafez avenue 15194, Tehran, Iran. E-mail: \textit{a.s.daghighi@gmail.com}, Website: http://alidaghighi.org/.\\
$^2$Department of Mathematics and Computer Science, Amirkabir University of Technology, Hafez avenue 15194, Tehran, Iran. E-mail: \textit{pourmahd@ipm.ir}, Website: http://math.ipm.ac.ir/$\sim$pourmahdian/
}
  \maketitle


 \begin{abstract}
Strengthening a result of Amir Leshem \cite{leshem}, we prove that the consistency strength of holding $GCH$ together with definable tree property for all successors of regular cardinals is precisely equal to the consistency strength of existence of proper class many $\Pi^{1}_1$ - reflecting cardinals. Moreover it is proved that if $\kappa$ is a supercompact cardinal and $\lambda> \kappa$ is measurable, then there is a generic extension of the universe in which $\kappa$ is a strong limit singular cardinal of cofinality $\omega, ~ \lambda=\kappa^+,$ and the definable tree property holds at $\kappa^+$. Additionally we can have $2^\kappa > \kappa^+,$ so that $SCH$ fails at $\kappa$.
 \end{abstract}

\section{Introduction}
The tree property for a regular cardinal $\kappa$ is the statement that there is no $\kappa$ - Aronszajn tree or equivalently every $\kappa$ - tree has a cofinal branch. In general constructing a model for tree property on a regular cardinal $\kappa$ is not trivial and needs large cardinal assumptions. The problem becomes even harder and needs stronger large cardinal assumptions when one tries to get tree property on several successive regular cardinals. In this direction we have: 

\begin{proposition}\label{results} The following results are known about tree property:

\item[(1)] (Konig) The tree property holds on $\aleph_0$.

\item[(2)] (Aronszajn) The tree property does not hold on $\aleph_1$.

\item[(3)] (Specker) For every infinite cardinal $\kappa$ if $\kappa^{<\kappa}=\kappa$ then the tree property does not hold on $\kappa^+$. Specially if $CH$ holds then $\aleph_2$ does not have the tree property.

\item[(4)] (Silver - Mitchell) The tree property on $\aleph_2$ is equiconsistent with the existence of a weakly compact cardinal.

\item[(5)] (Abraham) Assuming the consistency of a supercompact cardinal and a weakly compact above it, it is consistent to have tree property on both $\aleph_2$ and $\aleph_3$.

\item[(6)] (Magidor) The consistency of tree property on both $\aleph_2$ and $\aleph_3$ implies the consistency of "$0^\sharp$ \textit{exists}".

\item[(7)] (Cummings - Foremann) Assuming the existence of an $\omega$-sequence of supercompact cardinals, it is consistent that the tree property holds for all $\aleph_{n}$'s, $1<n<\omega$.

 \end{proposition}
 \begin{proof}
For (1), (2), (3) see \cite{jech}. (4) is proved in \cite{mitchell}. For (5) and (6) see \cite{abraham}. The result (7) is proved in \cite{foremann}.
 \end{proof}

An importnat point about the Aronszajn's result in proposition \ref{results} is the essential use of AC in his construction. Thus the existing $\aleph_1$ - Aronszajn tree is \textit{not} definable. Amir Leshem \cite{leshem} proved that assuming existence of a $\Pi_{1}^{1}$ - reflecting cardinal, it is consistent that a definable version of tree property (definition \ref{tree property}) holds on $\aleph_1$.
\begin{definition}\label{reflecting cardinals}
An inaccessible cardinal $\kappa$ is $\Pi_{n}^{m}$ - reflecting, if for every $A\subseteq V_{\kappa}$ definable over $V_{\kappa}$ with parameters from $V_{\kappa}$ and for every $\Pi_{n}^{m}$ - sentence $\Phi$, if $(V_{\kappa}, \in, A)\models \Phi$ then there is an $\alpha<\kappa$ such that $(V_{\alpha}, \in, A\cap V_{\alpha})\models \Phi$.
\end{definition}

\begin{definition}\label{tree property}
Let $\kappa$ be a regular cardinal. A $\kappa$ - tree $(T, <_T)$ is definable if its underlying set is $\kappa$, and the relation $<_T$ is $\Sigma_n$ - definable in the structure $(H_{\kappa}, \in)$ for some natural number $n$.  We say the definable tree property holds on $\kappa$ if every definable $\kappa$ - tree has a cofinal branch.
\end{definition}
\begin{remark}\label{different forms of definable tree property}
In his paper \cite{leshem}, Leshem considers several variants of definable tree property, including what he calls definable tree property in the strict, wide and very wide sense. His results are about definable tree property in the strict sense which is exactly what we stated in the definition \ref{tree property}. According to Leshem's definitions, every definable $\kappa$ - tree in the strict sense is definable in the wide sense and every definable $\kappa$ - tree in the wide sense is definable in the very wide sense. Also every definable $\kappa$ - tree $(T, <_T)$ in the wide sense is isomorphic to a $\kappa$ - tree $(\kappa, <^*)$ that is definable in the strict sense. So it follows that without losing generality one can assume that the definable tree property in the strict and wide sense are identical while the definable tree property in the very wide sense is different from them.
\end{remark}
\begin{theorem}\label{leshem's main result} (Leshem)
The following statements are equiconsistent:
\begin{enumerate}
\item[(1)] The definable tree property holds on $\aleph_1$.
\item[(2)] There is a $\Pi_{1}^{1}$ - reflecting cardinal.
\end{enumerate}
\end{theorem}
\begin{proof}
\cite{leshem}.
\end{proof}
In section 2 we generalize Leshem's result to the consistency of definable tree property for proper class of all successors of regular cardinals using the existence of proper class many $\Pi_{1}^{1}$ - reflecting cardinals, a large cardinal assumption weaker than the existence of a Mahlo cardinal and much weaker than what is theoretically expected for achieving tree property in the usual sense for this class of regular cardinals.

\begin{main theorem}\label{main theorem}
The following statements are equiconsistent:
\begin{enumerate}
\item[(1)] The definable tree property on successor of every regular cardinal.
\item[(2)] There are proper class many $\Pi_{1}^{1}$ - reflecting cardinals.
\end{enumerate}
\end{main theorem}

The situation for the consistency of holding tree property at successor of a singular cardinal is generally more complicated than the case of regulars. By a result of Magidor and Shelah \cite{magidor} it is known that if $\lambda$ is the singular limit of $\lambda^+$ - supercompact cardinals then $\lambda^+$ has the tree property. This fact is used by them to prove the consistency of tree property on $\aleph_{\omega+1}$ from a very strong large cardinal assumption. Later Sinapova \cite{sinapova} decreased the necessary large cardinal assumption for proving the consistency of tree prperty on $\aleph_{\omega+1}$ to the existence of $\omega$ - many supercompact cardinals.

On the other hand, answering an old question of Woodin, Neeman \cite{neeman} produced, assuming the existence of $\omega$-many supercompact cardinals, a model in which $SCH$ fails at a singular strong limit cardinal $\kappa$ of cofinality $\omega$ and $\kappa^+$ has the tree property.  But in Neeman's model, $GCH$ fails cofinally often below $\kappa$, and it is still an open problem if we can have a singular cardinal $\kappa$ such that $GCH$ holds below $\kappa$, $2^\kappa > \kappa^+$, and $\kappa^+$ has the tree property.

In section 3 we prove the main theorem \ref{definable tree property at a singular} which gives an affirmative answer to this question if the tree property is replaced with the definable tree property. Our proof also reduces the large cardinal strength from the existence of infinitely many supercompact cardinals to the existence of a supercompact cardinal and a measurable above it.

\begin{main theorem}\label{definable tree property at a singular} 
Assume $GCH$ holds, $\kappa$ is supercompact and $\lambda >\kappa$ is measurable. Then there is a generic extension of the universe in which:
\begin{enumerate}
\item[(1)] $\kappa$ is a strongly limit singular cardinal of cofinality $\omega$,
\item[(2)] No bounded subsets of $\kappa$ are added, in particular $GCH$ holds below $\kappa$,
\item[(3)] $\lambda=\kappa^+$ and the definable tree property holds at $\lambda$,
\item[(4)] $2^\kappa = |j(\lambda)|$, in particular if (in $V$) $|j(\lambda)|> \lambda^+$, then $SCH$ fails at $\kappa$.
\end{enumerate}
\end{main theorem}

The generic extension in which the above theorem holds is essentially the extension obtained by supercompact extender based Prikry forcing introduced by Merimovich in \cite{mer4}. 

Our results show that the \textit{definable} version of tree property is so different in nature from its original form and needs much weaker large cardinal assumptions for proving its consistency.

\section{Definable tree property at successor of all regular cardinals}
The entire argument in this section is for proving the main theorem \ref{main theorem}.
\subsection{From definable tree property to reflecting cardinals} \label{section for from definable tree property to reflecting cardinals}
In this subsection we prove the (1) to (2) part of the main theorem \ref{main theorem} by showing that assuming definable tree property for successors of regular cardinals in $V$, $\Pi_{1}^{1}$ - reflecting cardinals form an unbounded subclass of cardinals in $L$ (theorem \ref{from tree property to reflecting}). First let's review some facts and definitions from \cite{leshem}.

\begin{definition}\label{extension property}
A cardinal $\kappa$ has the extension property if and only if for every natural number $n$ and for every set $A\subseteq V_{\kappa}$ definable over $V_{\kappa}$ with parameters from $V_{\kappa}$, there is a transitive set $X$, and a subset $A^X$ of $X$ such that $\kappa\in X$ and $(V_{\kappa}, \in, A)\prec_{n} (X, \in, A^X)$.
\end{definition}

\begin{proposition}\label{extension and end extension}
For a cardinal $\kappa$ the following statements are equivalent:
\begin{enumerate}
\item[(1)] $\kappa$ has the extension property.
\item[(2)] For every natural number $n$, there is a transitive set $X$ which $\kappa\in X$ and the structure $(X,\in)$ is a $\Sigma_n$ - elementary end extension of $(V_{\kappa}, \in)$.
\end{enumerate}
\end{proposition}
\begin{proof}
\cite{leshem}.
\end{proof}

\begin{proposition}\label{reflecting and extension}
For a cardinal $\kappa$ the following statements are equivalent:
\begin{enumerate}
\item[(1)] $\kappa$ is $\Pi_{1}^{1}$ - reflecting.
\item[(2)] $\kappa$ is inaccessible and has the extension property.
\end{enumerate}
\end{proposition}
\begin{proof}
\cite{leshem} theorem 3.2.
\end{proof}

\begin{proposition}\label{reflecting and tree property}
The definable tree property holds on every $\Pi_{1}^{1}$ - reflecting cardinal.
\end{proposition}
\begin{proof}
\cite{leshem} lemma 3.3.
\end{proof}

\begin{lemma}\label{tree property in V and reflecting in L}
Let $\kappa$ be a successor of a regular cardinal, if $\kappa$ has the definable tree property in $V$ then $\kappa$ is $\Pi_{1}^{1}$ - reflecting in $L$.
\end{lemma}
\begin{proof}
Similar to the proof of theorem 5.1. in \cite{leshem}.
\end{proof}

\begin{theorem}\label{from tree property to reflecting}
If the definable tree property holds for proper class many regular cardinals in $V$ then there are proper class many $\Pi_{1}^{1}$ - reflecting cardinals in $L$.
\end{theorem}
\begin{proof}
Assume that $\Pi_{1}^{1}$ - reflecting cardinals in $L$ are bounded below a cardinal $\lambda$. There is a regular cardinal $\kappa > \lambda$ such that definable tree property holds for $\kappa$ in $V$. By lemma \ref{tree property in V and reflecting in L}, $\kappa$ is a $\Pi_{1}^{1}$ - reflecting cardinal in $L$ greater than $\lambda$, a contradiction.
\end{proof}

\subsection{From reflecting cardinals to definable tree property}\label{section for from reflecting cardinals to definable tree property}

In this subsection we are going to prove the (2) to (1) part of the theorem \ref{main theorem} using an Easton reverse iteration of Levy collapses of reflecting cardinals (theorem \ref{from reflecting to tree property}). At the first setp we need to prove that small forcings preserve the $\Pi_{1}^{1}$ - reflecting cardinals.

\begin{lemma}\label{preservation of reflecting cardinals under small forcings}
If $\kappa$ is a $\Pi_{1}^{1}$ - reflecting cardinal and $\mathbb{P}$ is a notion of forcing which $|\mathbb{P}|<\kappa$ then $\kappa$ remains $\Pi_{1}^{1}$ - reflecting in $V^{\mathbb{P}}$.
\end{lemma}
\begin{proof}
Assume that $\kappa$ is a $\Pi_{1}^{1}$ - reflecting cardinal and $|\mathbb{P}|<\kappa$. As small forcings preserve inaccessibility of $\kappa$, by proposition \ref{reflecting and extension} it suffices to show that $\kappa$ has the extension property in $V[G]$. Using the equivalence in proposition \ref{extension and end extension} it suffices to show that in $V[G]$ for every natural number $n$, there is a transitive set $Y$ such that $\kappa\in Y$ and the structure $(Y,\in)$ is a $\Sigma_n$ - elementary end extension of $(V_{\kappa}, \in)$. Note that by smallness of forcing notion we have $V_{\kappa}^{V[G]}=V_{\kappa}[G]$. Thus it is sufficient to show that for every natural number $n$, there is a transitive set $Y\in V[G]$ which $\kappa\in Y$ and the structure $(Y,\in)$ is a $\Sigma_n$ - elementary end extension of $(V_{\kappa}[G], \in)$.

Fix the natural number $n$, without losing generality we may assume that the forcing notion $\mathbb{P}$ in $V$ is defined by a formula of complexity $\Sigma_{m}$. Choose the sufficiently large natural number $t\geq m, n$. By extension property of $\kappa$ in $V$ as a $\Pi_{1}^{1}$ - reflecting cardinal, we get a transitive set $X\in V$ and a set $\mathbb{P}^X\subseteq X$ such that $\kappa\in X$ and the structure $(X, \in, \mathbb{P}^X)$ is a $\Sigma_{t}$ - elementary extension of $(V_{\kappa}, \in, \mathbb{P})$. In fact $\mathbb{P}^X=\mathbb{P}$ because by elementary extension the structure $(X, \in, \mathbb{P}^X)$ is agree with $(V_{\kappa}, \in, \mathbb{P})$ on the notion of $\in$.

Now we show that $(V_{\kappa}[G], \in)\prec_{n}(X[G], \in)$ which completes the proof because $X[G]$ is a transitive set in $V[G]$ with our required property for $Y$. In order to do this fix a first order $\Sigma_{n}$-formula $\varphi (x_1,\cdots, x_n)$. We have $V[G]\models \varphi (a_1,\cdots, a_n)$ iff $\exists p\in G~~~p\Vdash_{\mathbb{P}}^{V}\varphi (\dot{a}_1,\cdots, \dot{a}_n)$. Note that by smallness of forcing we may assume that $\mathbb{P}\in V_{\kappa}$ and so we can consider the forcing relation $\Vdash^{V}$ as $\Vdash^{V_{\kappa}}$, thus the last statement is equivalent to $\exists p\in G~~~(V_{\kappa}, \in, \mathbb{P})\models p\Vdash_{\mathbb{P}}\varphi (\dot{a}_1,\cdots, \dot{a}_n)$. As $t$ was chosen sufficiently large we may assume that it exceeds the complexity of the formula  $p\Vdash_{\mathbb{P}}\varphi (\dot{a}_1,\cdots, \dot{a}_n)$ which is a $\Sigma_{s}$ - formula like $\psi_{\varphi} (p, \mathbb{P}, \dot{a}_1,\cdots, \dot{a}_n)$. Thus by $\Sigma_{t}$ - elementary extension, $\exists p\in G~~~(V_{\kappa}, \in, \mathbb{P})\models p\Vdash_{\mathbb{P}}\varphi (\dot{a}_1,\cdots, \dot{a}_n)$ holds iff $\exists p\in G~~~(X, \in, \mathbb{P})\models p\Vdash_{\mathbb{P}}\varphi (\dot{a}_1,\cdots, \dot{a}_n)$. Equivalently $X[G]\models \varphi (a_1,\cdots, a_n)$ which means $(V_{\kappa}[G], \in)\prec_{n}(X[G], \in)$ and so $\kappa$ is a $\Pi_{1}^{1}$ - reflecting cardinal in $V[G]$.
\end{proof}

We need to work with the notion of a weakly homogenous forcing that is defined as follows:

\begin{definition}\label{homogeneous forcing}
A notion of forcing $\mathbb{P}$ is called weakly homogeneous if and only if for every two conditions $p, q$ in $\mathbb{P}$ there is an automorphism $\pi$ of $\mathbb{P}$ such that $\pi (p)$ and $q$ are compatible.
\end{definition}

An important property of weakly homogeneous forcings is that they don't add any new definable set with parameters from the ground model.

\begin{lemma}\label{homogeneous forcing adds no definable}
Let $V[G]$ be a forcing extension of $V$ by a weakly homogeneous forcing notion and $S\in V[G]$ is a subset of $V$ definable in $V[G]$ using parameters from $V$. Then $S\in V$.
\end{lemma}
\begin{proof}
\cite{jech forcing} proposition 2.2.
\end{proof}

The next observation is that $\kappa^+$ - closed weakly homogeneous forcings preserve definable tree property on $\kappa^+$.

\begin{lemma}\label{preservation of definable tree property under closed homogeneous forcings}
If definable tree property holds on $\kappa^+$ and $\mathbb{P}$ is a $\kappa^+$ - closed weakly homogeneous notion of forcing then in $V^\mathbb{P}$, $\kappa^+$ has the definable tree property.
\end{lemma}
\begin{proof}
Assume the definable tree property holds on $\kappa^+$ in $V$ and $T$ is a $\kappa^+$ - tree in $V[G]$ which is definable in the structure $(H_{\kappa^+}^{V[G]}, \in)$. Thus there is a first order formula with parameters from $H_{\kappa^+}^{V[G]}$ which defines $T$. By $\kappa^+$ - closure of forcing we have $H_{\kappa^+}^{V[G]}=H_{\kappa^+}^{V}$ and so $T$ is definable in $V[G]$ with parameters from $V$. Thus by homogeneity of forcing $\mathbb{P}$ and lemma \ref{homogeneous forcing adds no definable}, $T\in V$.

$T$, $dom(<_T)$, $ran(<_T)$ are sets of ordinals. All these sets are definable in $V[G]$ and so lie in $V$. Since for homogeneous forcings every set of ordinals definable in $V[G]$ with parameters from $V$, then both $T$ and $<_T$ are definable in $V$ as well.

Now by $\kappa^+$ - closure property of forcing we know that cardinals $\leq \kappa^+$ are preserved and so $T$ is a $\kappa^+$ - tree in the ground model. Consequently by definable tree property for $\kappa^+$ in $V$, $T$ has a cofinal branch $b$ in $V$. Again by $\kappa^+$ - closure of forcing, $b$ is a cofinal branch for $T$ in the generic extension too. So in $V[G]$ the definable tree property holds on $\kappa^+$.
\end{proof}

\begin{lemma}\label{generalization of leshem's result}
Let $\kappa$ be a regular cardinal and $\lambda>\kappa$ is a $\Pi_{1}^{1}$ - reflecting cardinal, then in $V^{Col(\kappa, <\lambda)}$ we have $\kappa^+ = \lambda$ and the definable tree property holds on $\kappa^+$.
\end{lemma}
\begin{proof}
A straightforward modification of the proof of theorem \ref{leshem's main result}.
\end{proof}

\begin{theorem}\label{from reflecting to tree property}
If there are proper class many $\Pi_{1}^{1}$ - reflecting cardinals in $V$, then there is a generic extension of $V$ by a weakly homogeneous forcing such that $GCH$ holds and successor of every regular cardinal has the definable tree property.
\end{theorem}
\begin{proof}
Let $\langle \kappa_\alpha: \alpha \in Ord \rangle$ be an increasing continuous sequence of cardinals such that $\kappa_0=\aleph_0$, and for each successor ordinal $\alpha, \kappa_\alpha$ is a $\Pi_{1}^{1}$ - reflecting cardinal and no $\kappa_\alpha,$ for limit ordinal $\alpha,$ is inaccessible (otherwise cut the universe).

Let  $\mathbb{P}=\langle\langle\mathbb{P}_{\alpha}~|~\alpha\leq Ord\rangle, \langle\dot{\mathbb{Q}}_{\alpha}~|~\alpha\in Ord\rangle\rangle$ be the reverse Easton iteration such that
\begin{enumerate}
\item[(1)] $\mathbb{P}_0$ is the trivial forcing,

\item[(2)] For $\alpha=0,$ or $\alpha$ a successor ordinal, $ \Vdash_{\alpha}$ " $\dot{\mathbb{Q}}_\alpha=\dot{C}ol(\kappa_\alpha, < \kappa_{\alpha+1})$ '',

\item[(3)] For limit ordinal $\alpha, \Vdash_{\alpha}$ " $\dot{\mathbb{Q}}_\alpha=\dot{C}ol(\kappa^+_\alpha, < \kappa_{\alpha+1})$ ".
\end{enumerate}
Our defined forcing notion has the following properties:

\begin{lemma}\label{properties of forcing}
Let $G$ be $\mathbb{P}$-generic over $V$. Then
\begin{enumerate}
\item[(1)] $CARD^{V[G]}= \{\kappa_\alpha: \alpha \in Ord  \} \cup \{ \kappa^+_{\alpha}: \alpha \in Ord, \alpha $ is a limit ordinal $\}$,

\item[(2)] If $\lambda$ is successor of a regular cardinal in $V[G],$ then  $\lambda=\kappa_{\alpha+1},$ for some $\alpha$,

\item[(3)] If $\alpha=0$ or $\alpha$ is a successor ordinal, then $\mathbb{P}\simeq \mathbb{P}_\alpha * \dot{\mathbb{P}}_{[\alpha, \infty)},$ where
$\Vdash_\alpha$ " $\dot{\mathbb{P}}_{[\alpha, \infty)}$ is $\kappa_\alpha$ - closed and weakly homogeneous ''.

\item[(4)] If $\alpha$ is a limit ordinal, then $\mathbb{P}\simeq \mathbb{P}_\alpha * \dot{\mathbb{P}}_{[\alpha, \infty)},$ where
$\Vdash_\alpha$ " $\dot{\mathbb{P}}_{[\alpha, \infty)}$ is $\kappa^+_{\alpha}$-closed and weakly homogeneous ''.

\item[(5)] $GCH$ holds in $V[G]$.
\end{enumerate}
\end{lemma}
\begin{proof}
The proof is standard. The homogeneity part follows from the work of Friedman-Dobrinen \cite{friedman}.
\end{proof}

Now note that in $V[G]$ the definable tree property holds for successor of every regular cardinal. To see this let $\lambda$ be the successor of a regular cardinal in $V[G]$. By part (2) of lemma \ref{properties of forcing}, there is an ordinal $\alpha$ such that $\lambda=\kappa_{\alpha+1}$. Then we have the following cases:

\textit{Case 1:} $\alpha=0$ or $\alpha$ is a successor ordinal.

\noindent As $\kappa_{\alpha+1}$ is a $\Pi_{1}^{1}$ - reflecting cardinal in $V$ and all steps of our forcing up to $\mathbb{P}_{\alpha}$ are small with respect to cardinal $\kappa_{\alpha+1}$, it follows from lemma \ref{preservation of reflecting cardinals under small forcings} that $\kappa_{\alpha+1}$ remains $\Pi_{1}^{1}$ - reflecting in $V^{\mathbb{P}_{\alpha}}$. By definition of our iteration, we force with $\dot{C}ol(\kappa_\alpha, < \kappa_{\alpha+1})$ in $V^{\mathbb{P}_{\alpha}}$. By lemma \ref{generalization of leshem's result}, $\lambda$ will have definable tree property in $V^{\mathbb{P}_{\alpha+1}}$. Also if we split our iteration at $\alpha$ as $\mathbb{P}\simeq \mathbb{P}_\alpha * \dot{\mathbb{P}}_{[\alpha, \infty)}$, then by part (3) of lemma \ref{properties of forcing} the tail forcing at step $\alpha$ is $\kappa_{\alpha}$ - closed and weakly homogeneous. If $\alpha$ is a successor ordinal like $\beta+1$ then by lemma \ref{preservation of definable tree property under closed homogeneous forcings} it follows that the already forced definable tree property on other successors of regular cardinals less than $\lambda$ which are in the form $\theta=\kappa_{\gamma+1}$ for some $\gamma<\beta$, won't be destroyed by tail forcing because it is weakly homogeneous and has enough closure. Also in the case $\alpha = 0$ there is no successor of a regular cardinal below $\lambda$ and so we have nothing to prove.

\textit{Case 2:} $\alpha$ is a limit ordinal.

\noindent Note that by continuity of the sequence $\langle \kappa_\alpha: \alpha \in Ord \rangle$, we have $\kappa_{\alpha}=sup\{\kappa_{\beta}~|~\beta<\alpha\}$. By smallness of forcing up to stage $\alpha$ with respect to $\kappa_{\alpha+1}$, $\kappa_{\alpha+1}$ remains $\Pi_{1}^{1}$ - reflecting in $V^{\mathbb{P}_{\alpha}}$. Thus the inequality $\kappa_{\alpha}<\kappa_{\alpha}^{+}<\kappa_{\alpha+1}$ holds in $V^{\mathbb{P}_{\alpha}}$. By definition of our iteration we force with $\dot{C}ol(\kappa^+_\alpha, < \kappa_{\alpha+1})$ in this stage which by lemma \ref{generalization of leshem's result} makes the definable tree property on $\kappa_{\alpha+1}$ true in $V^{\mathbb{P}_{\alpha+1}}$. By part (4) of lemma \ref{properties of forcing} if we split our iteration as $\mathbb{P}\simeq \mathbb{P}_\alpha * \dot{\mathbb{P}}_{[\alpha, \infty)}$, the tail forcing is weakly homogeneous and $\kappa_{\alpha}^{+}$ - closed which by lemma \ref{preservation of definable tree property under closed homogeneous forcings} is sufficient to preserve the already forced definable tree property on all successors of regular cardinals less than $\lambda=\kappa_{\alpha+1}$.
\end{proof}

\section{Definable tree property at successor of a singular cardinal}
In this section we give the proof of the main theorem \ref{definable tree property at a singular}.
\subsection{Supercompact extender based Prikry forcing}
In this subsection, we present Merimovich's supercompact extender based Prikry forcing which appeared in \cite{mer4}. We present it in some details as we need it for later use.
For each $\a<j(\gl)$ let
$\lambda_\a$ be minimal $\eta<\gl$ such that $\a < j(\eta),$
and let $E(\a) \subseteq P(\lambda)$ be defined by
\begin{center}
$A\in E(\a) \Leftrightarrow \a \in j(A).$
\end{center}
Note that each $E(\a)$ is a $\k$-complete ultrafilter on $\lambda$ and it has concentrated on $\lambda_\a$. Also let
\begin{center}
$i_\a: \text{V} \rightarrow \text{N}_\a \simeq \text{Ult(V}, E(\a)).$
\end{center}
Finally put
\begin{center}
$E=  \langle \langle E(\a): \a<j(\lambda)    \rangle, \langle  \pi_{\beta, \a}: \beta, \a< j(\gl), \a\in range(i_\beta)          \rangle \rangle$
\end{center}
to be the extender derived from $j$, where $\pi_{\beta, \a}: \lambda \rightarrow \lambda$ is such that $j(\pi_{\beta, \a})(\beta)=\a$ (such a $\pi_{\beta, \a}$ exists as $\a \in range(i_\beta)$). Let $i: \text{V} \rightarrow \text{N }\simeq \text{Ult(V, E)}$ be the resulting extender embedding. We may assume that $j=i.$
\begin{definition}
Let $d\in [j(\lambda)]^{<\lambda}$ be such that $\k, |d| \in d.$ Then $\nu \in \text{OB(d)}$ if the following conditions hold:
\begin{enumerate}
\item $\nu: \dom(\nu) \rightarrow \lambda,$ where $\dom(\nu) \subseteq d,$
\item $\k, |d| \in \dom(\nu),$
\item $|\nu| \leq \nu(|d|),$
\item $\forall \a<\l$ $(j(\a)\in \dom(\nu) \Rightarrow \nu(j(\a))=\a ),$
\item $\a\in \dom(\nu) \Rightarrow \nu(\a) < \gl_\a,$
\item $\a < \beta$ in $\dom(\nu) \Rightarrow \nu(\a) < \nu(\beta).$
\end{enumerate}
Also for $\nu_0, \nu_1 \in \text{OB(d)},$ set $\nu_0 < \nu_1$ if and only if
\begin{enumerate}
\item [(6)]$\dom(\nu_0) \subseteq \dom(\nu_1),$

\item [(7)] For all $\a\in \dom(\nu_0)\setminus j[\gl], \nu_0(\a) < \nu_1(\a).$
\end{enumerate}
\end{definition}
We now define the forcing notion $\PP^*(E, \kappa, \lambda)$ as follows:
\begin{definition}
$\PP^*(E, \kappa, \lambda)$ consists of all functions $f: d \rightarrow \lambda^{<\omega}$, where  $d\in [j(\lambda)]^{<\lambda}$, $\k, |d| \in d,$ and  such that

$(1)$ For any $j(\a) \in d, f(j(\a))=\langle \a \rangle,$

$(2)$ For any $\a\in d\setminus j[\gl], $ there is some $k<\omega$ such that
\begin{center}
$f(\a)= \langle  f_0(\a), \dots, f_{k-1}(\a)               \rangle \subseteq \gl_\a $
\end{center}
$\hspace{0.7cm}$ is a finite increasing subsequence of $\gl_\a.$ For $f, g\in \PP^*(E, \kappa, \lambda),$
\begin{center}
$f \leq^*_{\PP^*(E, \kappa, \lambda)} g \Leftrightarrow f \supseteq g.$
\end{center}
\end{definition}
\begin{remark}
$\langle \PP^*(E, \kappa, \lambda), \leq^*_{\PP^*(E, \kappa, \lambda)} \rangle \approx Add(\lambda, |j(\lambda)|).$
\end{remark}
\begin{definition}
Assume $d\in [j(\lambda)]^{<\lambda}$ and $\k, |d| \in d.$
 Let $T \subseteq OB(d)^{<\xi} (1<\xi \leq \omega)$ and  $n<\omega.$ Then

 $(1)$ $Lev_n(T)=T \cap \text{OB(d)}^{n+1},$

 $(2)$ $\Suc_T(\langle \rangle) = \Lev_0(T),$

 $(3)$ $\Suc_T(\langle \nu_o, \dots, \nu_{n-1}         \rangle)=\{\mu\in OB(d): \langle \nu_o, \dots, \nu_{n-1}, \mu \rangle \in T    \}.$
\end{definition}
\begin{definition}
Assume $d\in [j(\lambda)]^{<\lambda}$ and $\k, |d| \in d.$
Let $T \subseteq OB(d)^{<\xi} (1<\xi \leq \omega)$. For $\langle \nu \rangle \in T,$ let
\begin{center}
$T_{\langle \nu \rangle}=\{  \langle \nu_o, \dots, \nu_{k-1} \rangle: k<\omega, \langle \nu, \nu_o, \dots, \nu_{k-1} \rangle \in T \}$
\end{center}
 and define by recursion for $\langle \nu_o, \dots, \nu_{n-1} \rangle \in T,$
\begin{center}
$T_{\langle \nu_o, \dots, \nu_{n-1} \rangle}= (T_{\langle \nu_o, \dots, \nu_{n-2} \rangle})_{\langle \nu_{n-1} \rangle}.$
\end{center}
\end{definition}
\begin{definition}
Assume $d\in [j(\lambda)]^{<\lambda}$ and $\k, |d| \in d.$
We define the measure $E(d)$ on $\text{OB(d)}$ by
\begin{center}
 $E(d)=\{ X \subseteq \text{OB(d)}: \text{mc(d)}\in j(X) \},$
\end{center}
where $\text{mc(d)}=\{\langle j(\a), \a \rangle : \a\in d   \}.$
\end{definition}
\begin{definition}
Assume $d\in [j(\lambda)]^{<\lambda}$ and $\k, |d| \in  d.$ Let $T \subseteq \text{OB(d)}^{<\omega}$ be a tree.
 $T$ is called an $E(d)$-tree, if
\begin{enumerate}
\item $\forall \langle \nu_0, \dots, \nu_{n-1} \rangle \in T$ $(\nu_0 < \dots < \nu_{n-1}),$
\item $\forall \langle \nu_0, \dots, \nu_{n-1} \rangle \in T$ $(\Suc_T(\langle \nu_0, \dots, \nu_{n-1} \rangle)\in E(d)).$
\end{enumerate}
\end{definition}
\begin{definition}
Assume $c\in [j(\lambda)]^{<\lambda}$ and  $A \subseteq \text{OB(d)}^{<\omega}.$ Then
\begin{center}
 $A \upharpoonright c =\{\langle \nu_0 \upharpoonright c, \dots, \nu_{n-1} \upharpoonright c \rangle: n< \omega, \langle \nu_0, \dots, \nu_{n-1} \rangle \in A \}.$
 \end{center}
\end{definition}
\begin{remark}
For $f\in \PP^*(E, \kappa, \lambda),$ we use $\text{OB(f)}, E(f)$ and $\text{mc(f})$ to denote $\text{OB}(\dom(f)), E(\dom(f))$ and $\text{mc}(\dom(f))$ respectively.
\end{remark}
We are now ready to define our main forcing notion, $\PP(E, \kappa, \lambda).$
\begin{definition}
$p\in \PP(E, \kappa, \lambda)$ iff $p=  \langle f^p, A^p \rangle$ where

$(1)$ $f^p \in \PP^*(E, \kappa, \lambda),$

$(2)$ $A^p$ is an $E(f^p)$-tree.
\end{definition}
\begin{definition}
Let $p, q\in \PP(E, \kappa, \lambda).$ Then $p \leq^* q$ ($p$ is a Prikry extension of $q$) iff:

$(1)$ $f^p \leq^*_{\PP^*(E, \kappa, \lambda)} f^q,$

$(2)$ $A^p \upharpoonright \dom(f^q) \subseteq A^q.$
\end{definition}
\begin{definition}
Let $f\in \PP^*(E, \kappa, \lambda), \nu \in OB(f)$ and suppose $\nu(\k) > \max(f(\k)).$ Then $f_{ \langle \nu \rangle}\in \PP^*(E, \kappa, \lambda)$ has the same domain as $f$ and
\begin{center}
 $f_{ \langle \nu \rangle}(\a) = \left\{ \begin{array}{l}
       f(\a)^{\frown} \langle \nu(\a) \rangle  \hspace{1.1cm} \text{ if } \a\in \dom(\nu), \nu(\a) > \max(f(\a)),\\
       f(\a)  \hspace{2.5cm} \text{Otherwise}.
 \end{array} \right.$
\end{center}
Given $\langle \nu_0, \dots, \nu_{n-1} \rangle \in OB(f)^n$ such that $\nu_0(\k) > \max(f(\k))$ and $v_0 < \dots < \nu_{n-1},$ define $f_{\langle \nu_0, \dots, \nu_{n-1} \rangle}$ by recursion as
\begin{center}
$f_{\langle \nu_0, \dots, \nu_{n-1} \rangle}=(f_{\langle \nu_0, \dots, \nu_{n-2} \rangle})_{\langle \nu_{n-1} \rangle}.$
\end{center}
Let $p\in \PP(E, \kappa, \lambda),$ and suppose $\langle \nu_0, \dots, \nu_{n-1} \rangle \in A^p$ is such that $\nu_0(\k) > \max(f^p(\k))$ and $v_0 < \dots < \nu_{n-1}.$ Then
\begin{center}
$p_{\langle \nu_0, \dots, \nu_{n-1} \rangle}=\langle f^p_{\langle \nu_0, \dots, \nu_{n-1} \rangle}, A^p_{\langle \nu_0, \dots, \nu_{n-1} \rangle} \rangle.$
\end{center}
\end{definition}
\begin{remark}
Whenever the notation $\langle \nu_0, \dots, \nu_{n-1} \rangle$ is used, where $\nu_0, \dots, \nu_{n-1} \in OB(f),$ it is implicitly assumed $\nu_0(\k) > \max(f(\k))$ and $v_0 < \dots < \nu_{n-1}.$
\end{remark}
\begin{definition}
Let $p, q\in \PP(E, \kappa, \lambda).$ Then
\begin{center}
$p \leq q \Leftrightarrow \exists \langle \nu_0, \dots, \nu_{n-1} \rangle \in A^q$ $(p \leq^* q_{\langle \nu_0, \dots, \nu_{n-1} \rangle}).$
\end{center}
\end{definition}
Let us state the main properties of the forcing notion $\PP(E, \kappa, \lambda).$ The proof can be found in \cite{mer4}.
\begin{theorem}\label{properties of prikry forcing}
Let $\text{G}$ be $\PP(E, \kappa, \lambda)$-generic over $\text{V}$. Then
\begin{enumerate}
\item $\langle \PP(E, \kappa, \lambda), \leq \rangle$ satisfies the $\l^+-c.c.,$

\item   $\langle \PP(E, \kappa, \lambda), \leq, \leq^* \rangle$ satisfies the Prikry property,

\item  $\langle \PP(E, \kappa, \lambda), \leq^* \rangle$ is $\k$-closed,

\item $cf^{\text{V[G]}}(\k)=\omega,$

\item All $\text{V}$-cardinals in the interval $(\k, \lambda)$ are collapsed,

\item $\lambda$ is preserved in $\text{V[G]},$

\item In $V[G], 2^\kappa=|j(\lambda)|.$
\end{enumerate}
\end{theorem}
It follows that $\text{V}$ and $\text{V[G]}$ have the same bounded subsets of $\k$ and $(\k^+)^{\text{V[G]}}=\lambda.$

\subsection{Projection of forcing notions}
Recall that we assumed $\lambda> \kappa$ is a measurable cardinal. Let $i: V \rightarrow N$ witnesses this; so $crit(i)=\lambda$ and
$^{\lambda}N \subseteq N.$ Consider the forcing notions $\PP(E, \kappa, \lambda)$
and $i(\PP(E, \kappa, \lambda)).$ Also note that
by closure of $N$ under $\lambda$-sequences, we have
\[
\PP(E, \kappa, \lambda)=\PP(E, \kappa, \lambda)_N,
\]
also it is clear that
\[
i(\PP(E, \kappa, \lambda)) = \PP(i(E), \kappa, i(\lambda))_N.
\]
Now by working in $N$, define $\pi: i(\PP(E, \kappa, \lambda)) \rightarrow \PP(E, \kappa, \lambda)$ as follows: let $p=\langle f^p, A^p \rangle \in i(\PP(E, \kappa, \lambda)).$ Set
\[
\pi(p)= \langle   f^p \upharpoonright (\dom(f^p) \cap j(\lambda)), A^p \upharpoonright    (\dom(f^p) \cap j(\lambda))             \rangle.
\]
The next lemma can be proved easily.
\begin{lemma}\label{projection lemma}
(In N) $\pi$ is a projection of forcing notions, in fact

$(1)$ $\pi(1_{i(\PP(E, \kappa, \lambda)) })=1_{\PP(E, \kappa, \lambda)},$

$(2)$ $\pi$ is order preserving with respect to both $\leq$ and $\leq^*$ relations,

$(3)$ If $p \in \PP(E, \kappa, \lambda), q \in i(\PP(E, \kappa, \lambda))$ and $p \leq \pi(q),$ then there exists $q^* \leq q$
such that $\pi(q^*) \leq^* p.$
\end{lemma}
\begin{proof}
Parts $(1)$ and $(2)$ can be proved easily, so we prove the part $(3)$. Thus let $p \in \PP(E, \kappa, \lambda), q \in i(\PP(E, \kappa, \lambda))$ and
suppose that $p \leq \pi(q).$ Let $q^*=\langle f^*, A^*        \rangle \in i(\PP(E, \kappa, \lambda))$ be such that:
\begin{enumerate}
\item $\dom(f^*)=\dom(f^p) \cup \dom(f^q),$
\item For $\alpha \in \dom(f^p), f^*(\alpha)=f^p(\alpha),$
\item For $\alpha \in \dom(f^q)\setminus \dom(f^p), f^*(\alpha)=f^q(\alpha),$
\item $A^*$ is an $E(f^*)$-tree,
\item $A^* \upharpoonright \dom(f^p) \subseteq A^p,$
\item $A^* \upharpoonright \dom(f^q) \subseteq A^q.$
\end{enumerate}
Then it is clear that $q^* \leq q$
and that $\pi(q^*) \leq^* p.$
The lemma follows.
\end{proof}

\subsection{Homogeneity of the quotient forcing}
Assume $H$ is $i(\PP(E, \kappa, \lambda))$-generic over $V$ and let $G$ be the filter generated by $\pi[H].$ By Lemma \ref{projection lemma} $G$ is $\PP(E, \kappa, \lambda)$-generic over $V$, and in $V[G]$, we can consider the quotient forcing:
\[
i(\PP(E, \kappa, \lambda))/ G = \{p \in i(\PP(E, \kappa, \lambda)): \pi(p) \in G     \}.
\]
In the next lemma we show that the above forcing has enough homogeneity properties. We will use this to show that some objects which are in $V[H]$
were already in $V[G].$
For a forcing notion $\PP$ and a condition $p\in \PP,$ set $\PP \downarrow p =\{q\in \PP: q \leq p   \}$ consists of all extensions of $p$ in $\PP.$ The homogeneity of our quotient forcing follows from the next theorem.
\begin{lemma}\label{homogeneity lemma} (Homogeneity lemma) 
Suppose $p, q \in i(\PP(E, \kappa, \lambda))$ so that $\pi(p)=\pi(q).$ Then there are $p^* \leq p, ~ q^* \leq q$ and an isomorphism
\[
\Phi: i(\PP(E, \kappa, \lambda))\downarrow p^* \cong i(\PP(E, \kappa, \lambda)) \downarrow q^*.
\]
\end{lemma}
\begin{proof}
Let  $p_1 \leq p$ and $q_1 \leq q$ be such that
\begin{enumerate}
\item $\dom(f^{p_1})=\dom(g^{q_1}),$ call it $d$,
\item $A^{p_1} = A^{q_1},$ call it $A$.
\end{enumerate}
For each $n<\omega$ and every $\langle \nu_0, \dots, \nu_{n-1}    \rangle \in A$ let $T ( \nu_0, \dots, \nu_{n-1}) \subseteq A_{\langle \nu_0, \dots, \nu_{n-1}    \rangle}$
 be such that 
for all $\langle \nu \rangle \in T ( \nu_0, \dots, \nu_{n-1})$ and all $\alpha \in \dom(\nu),$
\[
\nu(\alpha) > \max(f^{p_1}(\alpha)) \Leftrightarrow \nu(\alpha) > \max(f^{q_1}(\alpha)).
\]
By Lemma 3.12 \cite{mer4}, there are $p^* \leq^* p_1$ and $q^* \leq^* q_1$ such that
\begin{enumerate}
\item [(3)] $f^{p^*}=f^{p_1}$ and $f^{q^*}=f^{q_1}$,
\item [(4)] For each $n<\omega$ and $\langle \nu_0, \dots, \nu_{n-1}    \rangle \in A^{p^*},$
\begin{center}
$p^*_{\langle \nu_0, \dots, \nu_{n-1}    \rangle} \leq^* \langle f^{p_1}_{\langle \nu_0, \dots, \nu_{n-1}    \rangle}, T( \nu_0, \dots, \nu_{n-1})             \rangle$,
\end{center}
\item [(5)] For each $n<\omega$ and $\langle \nu_0, \dots, \nu_{n-1}    \rangle \in A^{q^*},$
\begin{center}
$q^*_{\langle \nu_0, \dots, \nu_{n-1}    \rangle} \leq^* \langle f^{q_1}_{\langle \nu_0, \dots, \nu_{n-1}    \rangle}, T( \nu_0, \dots, \nu_{n-1})                 \rangle$.
\end{center}
\end{enumerate}

We now define an isomorphism $\Phi$ from $i(\PP(E, \kappa, \lambda))\downarrow p^*$ onto $i(\PP(E, \kappa, \lambda))\downarrow q^*$
as follows: Assume $r \in i(\PP(E, \kappa, \lambda))$ and $r \leq p^*.$ Let $\Phi(r) \in i(\PP(E, \kappa, \lambda))$ be such that
\begin{enumerate}
\item [(6)] $\dom(f^{\Phi(r)})=\dom(f^r),$

\item [(7)] $\forall \alpha \in \dom(f^r)\setminus \dom(f^{p^*}), f^{\Phi(r)}(\alpha)=f^r(\alpha),$

\item [(8)] $\forall \alpha \in  \dom(f^{p^*}), f^{\Phi(r)}(\alpha)=f^{q^*}(\alpha) \cup (f^r(\alpha) \setminus f^{p^*}(\alpha)),$

\item [(9)] $A^{\Phi(r)}=A^r.$
\end{enumerate}
By our choice of $T( \nu_0, \dots, \nu_{n-1})$'s, $\Phi(r)$ is well-defined and it extends $q^*,$ so $\Phi(r) \in i(\PP(E, \kappa, \lambda))\downarrow q^*$
and
\[
\Phi: i(\PP(E, \kappa, \lambda))\downarrow p^* \rightarrow i(\PP(E, \kappa, \lambda)) \downarrow q^*
\]
is well-defined. It is also easily seen that $\Phi$
is in fact an isomorphism. The lemma follows.
\end{proof}

\subsection{Completing the proof of main theorem \ref{definable tree property at a singular}}
Finally we are ready to complete the proof of theorem \ref{definable tree property at a singular}. Let $V[G]$ be the generic extension obtained by $\PP(E, \kappa, \lambda).$
By theorem \ref{properties of prikry forcing}, in $V[G],$ $\kappa$ is strong limit singular of cofinality $\omega$
and $\kappa^+=\lambda.$ Further if $|j(\lambda)|> \lambda,$ then $2^\kappa > \kappa^+$ in $V[G].$ So it suffices to show that the definable tree
property holds in $V[G]$ at $\kappa^+=\lambda.$

Note that $H^{V[G]}(\lambda)=H^{N[G]}(\lambda)$. Now let $T \in V[G]$ be a $\lambda$-tree which is definable in $H^{V[G]}(\lambda)$ using parameters from  $H^{V[G]}(\lambda)$.
Also consider the forcing $i(\PP(E, \kappa, \lambda))$, and let $H$ be $i(\PP(E, \kappa, \lambda))$-generic over $V$ so that $G$
is the filter generated by $\pi[H];$ this is possible as $\pi$ is a projection map. We have $i[G]=G \subseteq H,$ so we can lift $i$ to an elementary embedding
\[
i^*: V[G] \rightarrow N[H]
\]
which is definable in $V[H].$

Then $i^*(T) \in N[H]$ is an $i^*(\lambda)$-tree, and since $i^*(\lambda)=i(\lambda)> \lambda,$  we can take some $x \in i^*(T)_\lambda$, the $\lambda$-th level of $i^*(T)$.
Now consider
\[
b=\{y \in i^*(T): y <_{i^*(T)} x  \}.
\]
Then $b$ is a branch of $T$ which lies in $N[H] \subseteq V[H].$ But $b$ is definable in $V[H]$ using parameters from $V[G],$
and hence using the homogeneity lemma \ref{homogeneity lemma}, $b \in V[G].$ Thus $T$ has a cofinal branch in $V[G]$, and the result follows.

\section{Open questions}
We proved that the consistency strength of having definable tree property for successor of every regular cardinal is exactly the consistency strength of having proper class many  $\Pi_{1}^{1}$ - reflecting cardinals. As it is stated in the part (7) of proposition \ref{results}, the existing proof for the consistency of usual tree property for a much smaller subclass of successors of regular cardinals, namely $\{\aleph_n~|~1<n< \omega\}$, uses a very strong large cardinal assumption in order of $\omega$ - many supercompacts. We also decreased the large cardinal assumption necessary for proving the consistency of definable tree property at successor of a singular cardinal.

The question regarding the consistency and consistency strength of usual tree property for successors of all regular cardinals is still open. The questions related to the consistency of tree property for successors of all singular cardinals and also for all regular cardinals in general are also open. Inspired by these open problems regarding the usual tree property, the following similar questions about definable tree property arise:

\begin{question}
Is it consistent to have definable tree property for successor of every singular cardinal? What is the consistency strength of this statement?
\end{question}

\begin{question}
Is it consistent to have definable tree property for all regular cardinals? What is the precise consistency strength of it?
\end{question}

 \end{document}